\documentclass[11pt]{article}
\usepackage[english]{babel}
\usepackage[cp1252]{inputenc}
\usepackage{graphicx}
\usepackage{float}
\usepackage{amsmath, amssymb, amsthm,color}
\usepackage{indentfirst}
\newtheorem{Teo}{Theorem}[section]

\newtheorem{Prop}[Teo]{Proposition}
\newtheorem{Lema}[Teo]{Lemma}
\newtheorem{Cor}[Teo]{Corollary}

\begin{document}

\title{\textsc{Alexander-Conway Polynomial State Model and Link Homology}}

\author{Louis H. Kauffman\\
Mathematics Department\\
University of Illinois at Chicago. USA.\\
{\tt kauffman@uic.edu}\\
and\\
Marithania Silvero \footnote{Partially supported by MTM2010-19355, P09-FQM-5112 and FEDER.}\\
Departamento de Álgebra. \\
Universidad de Sevilla.
Spain.\\
{\tt marithania@us.es}\\
}

\maketitle


\noindent \textbf{Abstract} \,
This paper shows how the Formal Knot Theory state model for the Alexander-Conway polynomial is related to Knot Floer Homology. In particular we prove a parity result about the states in this model that clarifies certain relationships of the model with Knot Floer Homology.

\bigskip

\noindent \textbf{Keywords:} knot, link, state sum, Alexander-Conway polynomial, Knot Floer homology.
\vspace{0.5cm}
\mbox{ }


\section{Introduction}

In 1928 J. W. Alexander introduced in \cite{AlexanderOriginal} a link polynomial with integer coefficients, later known as the Alexander polynomial, as the determinant of a matrix associated with an oriented diagram representing a link. He proves the invariance of this polynomial, up to sign and a power of the variable, under Reidemeister moves. This was the only known link polynomial until 1984, when Jones discovered his polynomial.

\vspace{0.1cm}

Fifty five years after Alexander, Kauffman reformulated the Alexander polynomial as a state summation in \cite{LibroKauffman}. In other words, he gave a formula for the Alexander polynomial that is a sum of evaluations of combinatorial configurations (states) related to the link diagram. Actually he provides a combinatorial definition of the Conway version of the Alexander polynomial by using such states (see Theorem \ref{TeoKauff}). In particular, this version is invariant under Reidemeister moves (exactly, not up to sign or powers of the variable).

\vspace{0.1cm}

Knot Floer Homology is a relatively new link invariant developed by Ozsváth and Szábo \cite{OszvathHeegaardHolomDisksKnots} and Rasmussen \cite{RasmussenThesis}. Since then there have been many authors interested in it, as it provides additional geometric information about a link, such as its genus or its fibredness, and even characterizes the trivial knot. One of the most surprising facts about Knot Floer Homology is that it categorifies the Alexander polynomial. In fact, it categorifies the Alexander-Conway polynomial since that polynomial, for knots, can be identified with a normalized version of the Alexander polynomial (see the references above). This is also proved in the paper \cite{OszvathHegaardKauff} by seeing the chain complex generated by the states of the Formal Knot Theory model for the Alexander-Conway polynomial (but uses contact geometry to find the differentials).

\vspace{0.1cm}

In \cite{OszvathHeegaardAlternating,OszvathHegaardKauff}, the authors identify the graded Euler characteristic of their homology theory as a state summation related to the Formal Knot Theory model. There is a technical difference in that they sum over different local weights and use a different sign (related to a white hole rather than a black hole count). They claim that the resulting state sum is the Alexander-Conway polynomial by reference to the FKT model (Theorem 11.3 in \cite{OszvathHegaardKauff}). We read this and wondered if it were so. It is the purpose of this paper to prove that it is so!

\vspace{0.1cm}

This paper can be thought of as the completion of the proof given in \cite{OszvathHegaardKauff} that Knot Floer Homology categorifies the Alexander-Conway polynomial. 
In fact, this proof is completed in a different way in \cite{OszvathHeegaardHolomDisksKnots} where Ozsváth and Szábo prove that their Euler Characteristic for Knot Floer Homology satisfies the Conway skein
relation and normalizations. And this result has also been verified by different methods \cite{Cimasoni} using the purely combinatorial version of Knot Floer Homology, but that version does not contact the FKT model directly. What is not proved in \cite{OszvathHegaardKauff} is that the state summation proposed in that paper is equal to the 
original FKT state summation. That is what we accomplish in this paper.

\vspace{0.1cm}

At the level of combinatorial knot theory, this paper shows that a reformulation of the FKT model using white hole rather than black hole counts gives the same topological result. In proving this result we obtain an interesting result (Proposition \ref{PropParidad}) relating the parities of the white and black holes for the states in the FKT model. This result is of independent interest.

\vspace{0.1cm}

In the original FKT model, the formula for the Alexander-Conway polynomial has the form $$\nabla_L(t) = \displaystyle\sum_{S\in\mathcal{S}} (-1)^{\sharp B_S} <L|S>.$$ Here $<L|S>$ denotes a product of local vertex weights and $\sharp B_S$ denotes the number of black holes in the state $S$ (see Figure \ref{mark}). Letting $\mu(L)$ denote the number of components of the link $L,$
we show that $\sharp B_S + \sharp W_S \equiv \mu(L) + 1, \, \mbox{mod 2}$ (see Proposition \ref{PropParidad}). From this it follows that $$\nabla_L(t) = \displaystyle\sum_{S\in\mathcal{S}} (-1)^{\sharp W_S} [L|S],$$ where the weights are changed to their inverses. This is the formula that Ozsváth and Szabó used in their categorification. Our parity result confirms that Knot Floer Homology categorifies the Alexander-Conway polynomial.

\section{Parity of white and black holes}

Given an oriented link $L \in S^3$, fix an associated planar diagram $P_L$, and write $v_1, \ldots, v_n$ for its double points. We assume that the diagram $P_L$ is connected, as the Conway-Alexander polynomial of a disconnected diagram is equal to zero. Then $P_L$ divides the plane into $n+2$ regions, one of them unbounded. Choose two adjacent regions and mark each of them with a star. A state $S$ associated to $P_L$ is an assignment of markers to each $v_i$ (that is, a choice of an adjacent unstarred region for each double point) so that no region of $P_L$ contains more than one marker. Let $\mathcal{S}$ denote the set of all possible states in $P_L$. In \cite{LibroKauffman} Kauffman reformulates the Alexander Polynomial as a state summation in the following way:

\vspace{0.1cm}

In Figure \ref{mark} we indicate two types of labeling for the crossings of an oriented link diagram. The labeling shown in the first row will be called
the {\it Alexander-Conway labeling.} It is a labeling that gives rise to a normalized version of the Alexander polynomial and gives a state summation for the Alexander-Conway polynomial in $z = t^{1/2} - t^{-1/2}$, as will be shown in Theorem \ref{TeoKauff}. We call these Alexander-Conway labels because they are different from labels that derive directly from Alexander's original paper on his polynomial \cite{AlexanderOriginal}. The reader may be interested in the way in which this is formulated in  \cite{LibroKauffman} (Formal Knot Theory by Kauffman).  For this we refer the reader to the Figure on page 7 where the labels are $B$ and $W$. There it is understood that the unlabeled regions at the crossing are labeled 1 and that for the case of the Alexander-Conway polynomial, we take B = $t^{-1/2}$ and 
$W = t^{1/2}.$ If the reader examines the Dover reprint of the book, these conventions are explained in more detail in the article ``Remarks on Formal Knot Theory" that is included with the book.

\vspace{0.1cm}

Given a state $S$, let $<L|S>$ denote the product of the Alexander-Conway labels associated to each marker in $S$ (see Figure \ref{mark}).  A marker is called a black hole if it occurs between two ingoing lines to the corresponding crossing.  A marker is called a white hole if it occurs between outgoing lines at a crossing. Write $B_S$ ($W_S$) for the set of markers at crossings $v_i$ in $S$ that occupy  black (white) holes.

\begin{figure}
\centering
\includegraphics[width = 11cm]{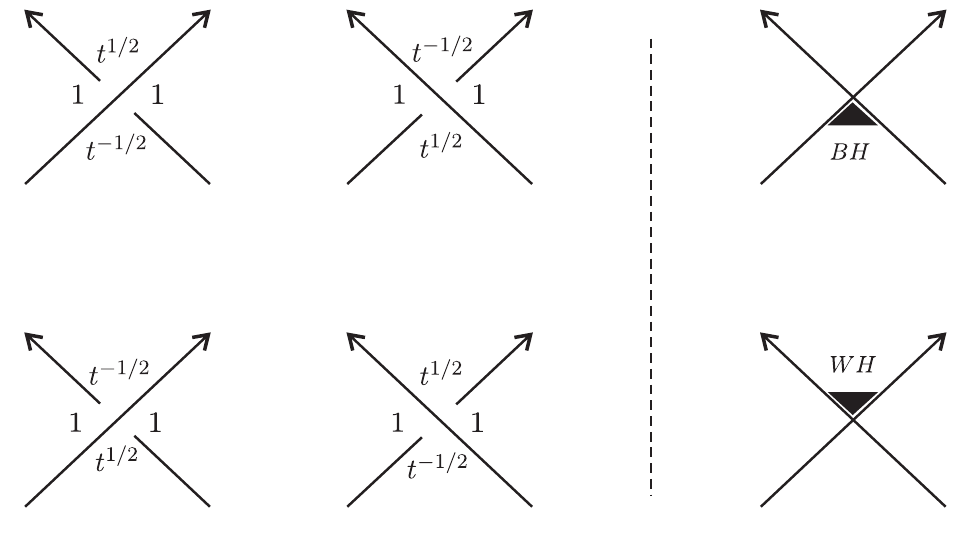}
\caption{\small{The left part of the first (second) row shows the Alexander-Conway labeling defining $<\cdot \, | \, \cdot>$ ($[\cdot \, | \, \cdot]$); the rightmost part shows the state marker producing a black (white) hole.}}
\label{mark}
\end{figure}

\vspace{0.1cm}

\begin{Teo}{\rm{\cite[Theorem 4.3]{LibroKauffman}}}\label{TeoKauff}
Following the notation above, write
$$
\nabla_L(t) = \sum_{S\in\mathcal{S}} (-1)^{\sharp B_S} <L|S>.
$$

Then, $\nabla_L(t)$ satisfies the following conditions: \\
1.- $\nabla_L(t) \stackrel{\cdot}{=} \Delta_L(t)$, with $\Delta_L(t)$ being the original version of the Alexander polynomial given in \cite{AlexanderOriginal} and $\stackrel{\cdot}{=}$ meaning equality up to sign and powers of $t$. \\
2.- If $L$ and $L'$ are equivalent oriented links, then $\nabla_L(t) = \nabla_{L'}(t)$. \\
3.- $\nabla_\bigcirc(t) = 1$.\\
4.- $\nabla_{L_+}(t) - \nabla_{L_-}(t) = (t^{1/2} - t^{-1/2}) \cdot \nabla_{L_0}(t)$, where $L_+, L_-$ and $L_0$ are identical oriented links differing only in a small neighborhood as shown in Figure \ref{D+-0}.\\
Hence, $\nabla_L(t)$ is the Alexander-Conway polynomial of the link $L$, that is, the Alexander polynomial normalized and such that $\nabla_\bigcirc(t) = 1$.
\end{Teo}

\begin{figure}[H] \label{D+-0}
\centering
\includegraphics[width = 5cm]{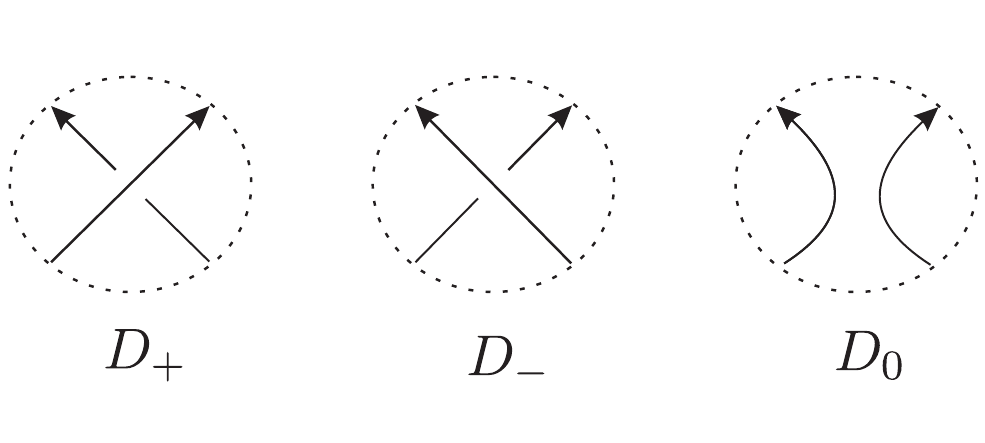}
\caption{}
\end{figure}

\vspace{0.1cm}

Note that it follows from the skein relation given here that $\nabla_L(t)$ is a polynomial in $z = t^{\frac{1}{2}} - t^{-\frac{1}{2}}$.

\vspace{0.1cm}

In \cite{OszvathHegaardKauff} an alternative definition of the Alexander polynomial is given in the following terms:

$$
\Theta_L(t) = \sum_{S\in\mathcal{S}} (-1)^{\sharp W_S} [L|S] \stackrel{\cdot}{=} \Delta_L(t),
$$
where the square bracket is defined in terms of the labels in the second row of Figure \ref{mark}. This definition is a state sum using white holes and labels that are essentially the inverses of the labels in the Theorem above.
\vspace{0.1cm}

We will show that this alternative definition of the Alexander polynomial is, in fact, equal to the Alexander-Conway polynomial as expressed in the original Formal Knot Theory state sum.

\begin{Teo}\label{ourTheorem}
The version given by Ozsváth and Szabó in \cite{OszvathHegaardKauff} of the Alexander polynomial and the Formal Knot Theory state sum given in \cite{LibroKauffman} are equal; that is
$$
\Theta_L(t) = \sum_{S\in\mathcal{S}} (-1)^{\sharp W_S} [L|S] \, = \, \sum_{S\in\mathcal{S}} (-1)^{\sharp B_S} <L|S> \, =  \nabla_L(t).
$$

This shows that the Knot Floer Homology of Ozsváth and Szabó categorifies the Alexander-Conway polynomial via the Formal Knot Theory state sum.
\end{Teo}

\noindent The proof relies on the following Proposition:

\begin{Prop} \label{PropParidad}
Let $L$ be an oriented link with $\mu(L)$ components and $S$ any FKT state in $L$. Then, the parities of the number of black and white holes in $S$ are equal when $\mu(L)$ is odd, and opposite when $\mu(L)$ is even.
\end{Prop}

\vspace{0.2cm}

\noindent {\bf Definition.} A {\it flat Reidemeister move} is a move on a planar diagram with crossings but no specification of over or under crossings. Such diagrams are called {\it universes} in  \cite{LibroKauffman}.
Flat Reidemeister moves are diagrammatically the same as the usual Reidemeister moves, but have no constraint about patterns of over and under crossings.\\

\begin{proof}[Proof of Proposition \ref{PropParidad}:]
\begin{Lema}\label{LemaPlanarReid}
For this lemma the Seifert surface specifically denotes the surface constructed from a given link diagram by using Seifert's algorithm.
Let $D$ be  a planar diagram representing an oriented link $L$; write $c(D)$ and $s(D)$ for the number of crossings and circles in the Seifert surface associated to $D$. The parity of $e(D) = c(D) - s(D) + 1$ is invariant under flat Reidemeister moves. Hence, the parity is a link invariant which will be called $E(L)$.
\end{Lema}

\begin{proof}
The proof is given in Figure~\ref{reidproof}.
\end{proof}

\begin{figure}
\centering
\includegraphics[width = 14cm]{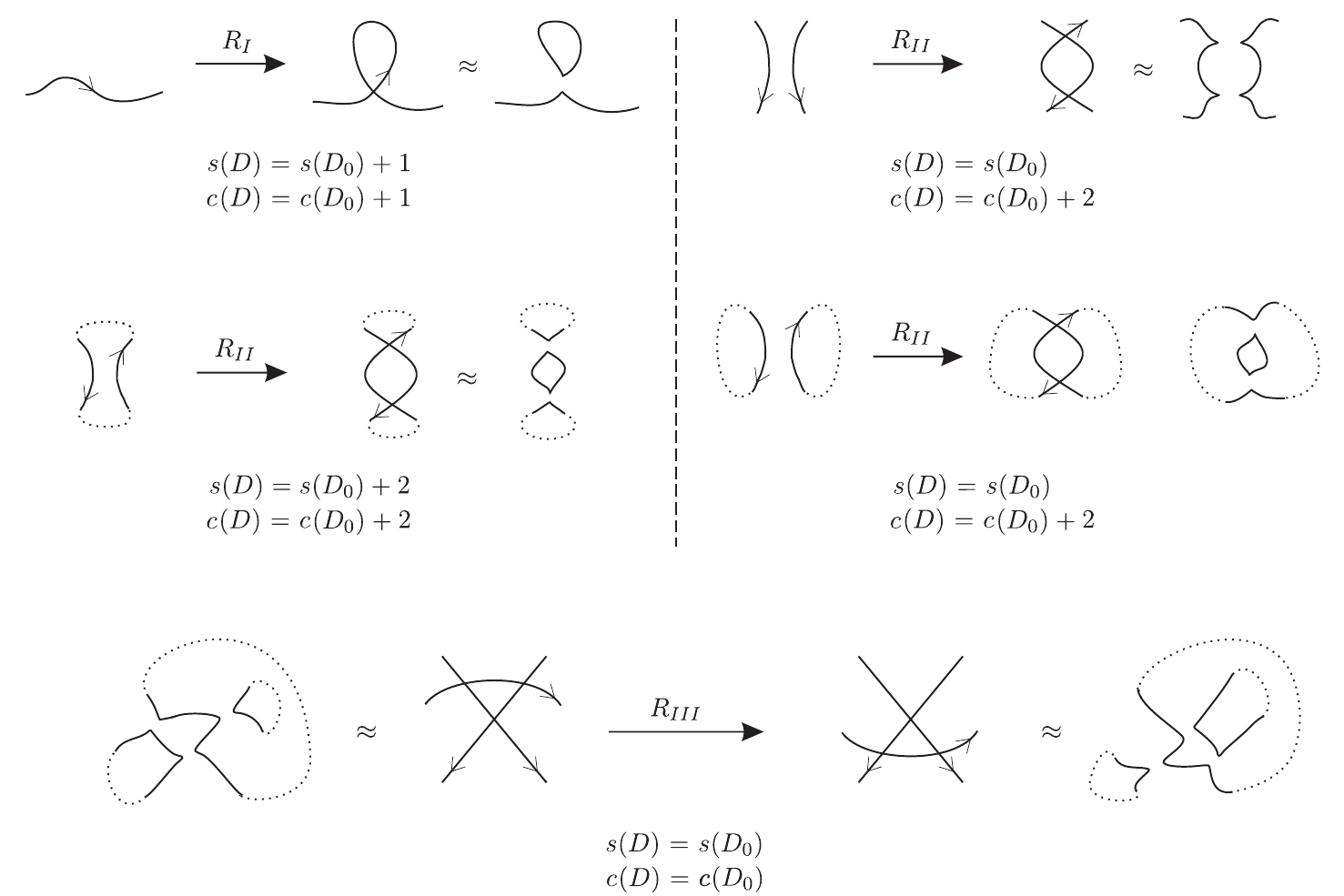}
\caption{\small{The proof of Lemma \ref{LemaPlanarReid} for first and second Reidemeister moves with all possible orientations and Seifert circles configurations are shown in this picture. For the third Reimeister move the many different possibilities can be checked as in the example.}}
\label{reidproof}
\end{figure}

\begin{Cor}
Let $L$ be an oriented link with $\mu(L)$ components. Then the parity of $\mu(L)$ is opposite to $E(L)$.
\end{Cor}

\begin{proof}
By a finite sequence of flat Reidemeister moves, any planar diagram representing $L$ can be transformed into the trivial link with $\mu(L)$ components, which satisfies $e(U_{\mu(L)}) = 1 - \mu(L)$.
\end{proof}

\vspace{0.2cm}

\begin{Lema}
Let $D$ be a planar diagram of an oriented link $L$. There always exists an FKT state $S$ verifying $\sharp B_S + \sharp W_S = c(D) - s(D) + 1$, where $B_S$ and $W_S$ are the number of black and white holes in $S$.
\end{Lema}

\begin{proof}
We give an algorithm for constructing such an state, following some ideas in \cite{LibroKauffman}. \\
Starting from $D$, construct its associated Seifert circles by smoothing every crossing in the only way that preserves orientation; the result will be a finite set of disjoint circles. Locate a Seifert circle with empty interior, choose one of its sites and reassemble it so that the resulting configuration has one fewer circle. After repeating this procedure $s(D) - 1$ times using each site at most once, we obtain a Jordan trail $T$.

\begin{figure}
\centering
\includegraphics[width = 16cm]{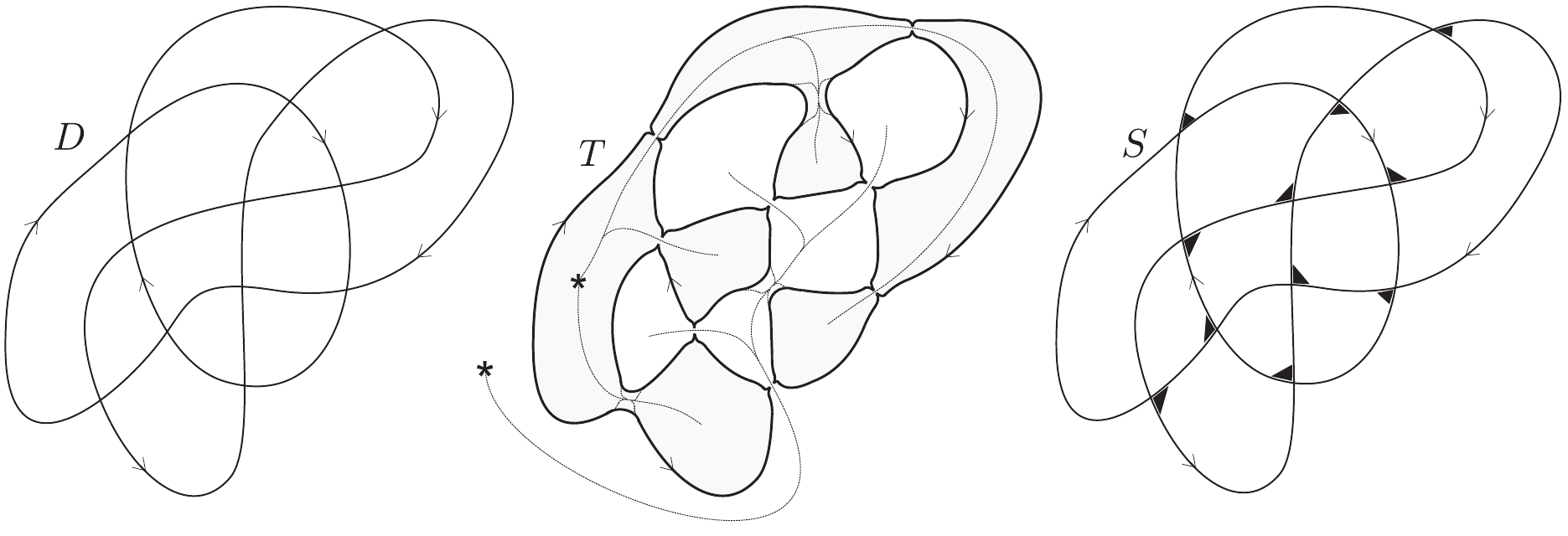}
\caption{\small{The figure shows the trail $T$ associated to the planar diagram $D$ together with a choice of the starred regions and the two rooted trees. It also shows the associated Kauffman state $S$.}}
 \label{DiagramtoTrailtoState}
\end{figure}

\vspace{0.1cm}

Kauffman proved that there exists a bijection between the collection of states in $D$ sharing a fixed choice of adjacent stars and the collection on all Jordan trails arising from $D$. Choose two adjacent regions in $D$ and draw an star in each of them; let $S$ be the Kauffman state associated to $T$ constructed in the following way: construct two trees, each one rooted at one of the stars, in such a way that their branches visit each region once (see Figure \ref{DiagramtoTrailtoState}). In each vertex draw a marker in the place where the branch leaves the site.

\vspace{0.1cm}

Let's see that $S$ verifies the condition in the statement. The markers coming from a re-assembling are neither black nor white holes. Each marker coming from a smoothing preserving orientation is either a black or a white hole. As we did $s(D) - 1$ re-assemblings, it follows that $\sharp(B_S) + \sharp(W_S) = c(D) - s(D) + 1$.

\end{proof}

The collection of possible states arising from a given planar diagram with a fixed choice of adjacent unmarked regions is a lattice, where all the states are related by a finite sequence of clock moves (see \cite{LibroKauffman} for a proof). As clock moves preserve the number of black plus white holes, every state $S$ arising from a diagram $D$ satisfies $$\sharp B_S + \sharp W_S = c(D) - s(D) + 1.$$
Note that $c(D) - s(D) + 1$ is equal to
$2g + \mu(D) -1$, where $g$ is the genus of the surface constructed by
Seifert's algorithm on $D.$\\

\vspace{0.1cm}

This completes the proof of Proposition \ref{PropParidad}.
\end{proof}

\vspace{0.2cm}

\begin{proof}[Proof of Theorem \ref{ourTheorem}:] \mbox{ }\\

Let $\overline{L}$ denote the mirror image of $L$. Since $\nabla_L(t)$ satisfies the skein relation in Theorem~\ref{TeoKauff}, after substituting $z = t^{\frac{1}{2}} - t^{-\frac{1}{2}}$ and using an inductive argument, it follows

$$\nabla_{L}(z) = \nabla_{\overline{L}}(-z) = (-1)^{\mu(L)+ 1} \nabla_{\overline{L}}(z).$$

\vspace{0.1cm}

Note that the labels used in $[\cdot | \cdot]$ for positive (negative) crossings are equal to those used in $<\cdot | \cdot>$ for negative (positive) crossings. As a consequence

$$\nabla_{\overline{L}}(z) \, =  \sum_{S\in\mathcal{S}} (-1)^{\sharp B_S} <\overline{L} | S> \, = \, \sum_{S\in\mathcal{S}} (-1)^{\sharp B_S} [L | S].$$

\vspace{0.1cm}

Finally, Proposition \ref{PropParidad} relates the number of black and white holes in a given state $S$ in the following way:
$$(-1)^{\sharp W_S} = (-1)^{\mu(L) + 1} (-1)^{\sharp B_S}.$$

\vspace{0.1cm}

The result holds after combining these facts:

$$\sum_{S\in\mathcal{S}} (-1)^{\sharp B_S} <L | S> \, = \, \nabla_{L}(t) \, = \, (-1)^{\mu(L)+ 1} \nabla_{\overline{L}}(t) \, = $$
$$ = (-1)^{\mu(L)+ 1} \sum_{S\in\mathcal{S}} (-1)^{\sharp B_S} [L | S] \, = \, \sum_{S\in\mathcal{S}} (-1)^{\mu(L)+ 1} (-1)^{\sharp B_S} [L | S] = $$
$$ = \sum_{S\in\mathcal{S}} (-1)^{\sharp W_S} [L | S] = \Theta_L(t).$$

\end{proof}

\vspace{0.7cm}

\noindent {\bf Remark.} Our parity result and the corresponding identities about the state sums applies to both knots and links. To our knowledge, the formulation of Knot Floer Homology using the FKT states is published only for knots. Presumably our result for links will match a corresponding formulation of Link Floer Homology using the FKT states.
This is a problem for further research.
\bigbreak



\bibliographystyle{plain}
\bibliography{Bibliograf}

\end{document}